\documentclass{article}
\usepackage{graphicx} 
\usepackage{amsmath}
\usepackage{amsthm}
\usepackage{amsfonts}
\usepackage{amssymb}
\usepackage{mathtools}
\newtheorem{theorem}{Teorema}
\newtheorem{lemma}[theorem]{Lemma}
\newtheorem{proposition}[theorem]{Proposition}
\newtheorem{corollary}[theorem]{Corollary}
\newtheorem{definition}{Definición}

\newtheorem{question}{Question}
\usepackage{url}
\usepackage{comment}
\DeclareMathOperator{\rank}{rank}
\DeclareMathOperator{\supp}{supp}
\DeclareMathOperator{\Ann}{Ann}
\DeclareMathOperator{\wt}{wt}
\usepackage{comment}

\usepackage{framed}
\usepackage{xfrac}

\allowdisplaybreaks

\title{Three results on twisted $G-$codes and skew twisted $G-$codes}
\author{Alvaro Otero Sanchez, aos073@ual.es}
\date{December 2025}

\begin{document}

\maketitle

\begin{abstract}
In this paper, we solve an open question posed in the original paper on twisted skew group codes concerning the conditions under which a twisted skew group code is checkable. We also prove that all ideals of dimension three over a twisted group algebra are abelian group codes, thereby generalizing a previous result for group algebras. Finally, we establish a bound on the dimension and minimum distance of a twisted group code, and characterize the cases in which this bound is attained.
\end{abstract}

\medskip


\medskip

\section{Introduction}
The foundational results in \cite{Shannon1948} and \cite{Hamming1950} started the research in Information Theory and Error Correcting code Theory. Since the early days of the field, several algebraic structures have been used to obtain linear codes with the necessary properties for each purpose. 

Let $G$ be a group and $K$ be a commutative ring. Then the group algebra $KG$ is the formal linear combination of elements of $G$, where product is given by the group operation in $G$. This ring was initially studied as they have a strong connection with representation theory, where linear representation of a group are in bijection with modules over the group algebras. 

In \cite{Berman1967} Berman proved that binary Reed Muller codes are ideals in a binary group algebra over an abelian group. In the following decade some mathematicians studied ideals over group algebras as they provide a rich framework to obtain new linear codes with interesting properties. 

This research area remains very active. For example, in \cite{BorelloDeLaCruzWillems2022}, the authors present the first bound relating the minimum distance, dimension of a linear code, and the underlying group. Another result connecting ring theory and coding theory is \cite{GARCIAPILLADO2019167}, where it is shown that group codes of dimensions two and three are commutative.

Recently, new generalizations of group algebras have been used in coding theory, such as  twisted group codes \cite{deLaCruzWillems2021}, skew group codes \cite{BoucherGeiselmannUlmer2007} or twisted skew G-codes \cite{Behajaina2024Twisted}. These structures provide interesting examples of linear codes. To illustrate this, the  ternary extended self-dual Golay code studied in \cite{deLaCruzWillems2021} are twisted group codes, or in \cite{Behajaina2024Twisted}, the authors present a $[6, 3, 4]_9$, where the authors point out that it is constructed in a simpler way that in \cite{Grassl}

In addition, there are still some open problems regarding the properties of such codes. In \cite{Behajaina2024Twisted}, the authors proved a sufficient condition for a twisted group code to be code checkable. However, they propose as a question if there is an equivalent of such result to the case of twisted skew group codes.

In this paper, some results regarding group codes or twisted group codes are generalized. To begin with, using an approach similar to that in the original paper, we give a positive answer to the question in \cite{Behajaina2024Twisted} about checkable codes in twisted skew group rings. After this, we give a sufficient condition for a twisted code to be permutation equivalent to a group code, which allows us to prove that all twisted codes of dimension 2 and 3 are abelian, and that may be useful in future with new twisted codes problems. Finally, in \cite{WillemsTwistedToAppear} a bound is proved on the dimension and minimal distance of an ideal over a twisted group algebra, and in \cite{Borello2022Ideals} group codes where the bound is reach are characterized. We present a generalization of this result to the case of twisted group codes.

\section{Preliminary results}
We start by recalling some basic background and the notation we will use throughout this work. For more information, see \cite{Karpilovsky1989}.

\begin{definition}
A crossed system is a quartet $(G,B,\sigma,\alpha)$ where $G$ is a group, $B$ is an $R$-algebra, being $R$ a ring and $\sigma$, $\alpha$ are applications
\begin{equation}
\sigma : G \rightarrow Aut(B) \hspace{1cm}
\alpha: G \times G \rightarrow U(B)
\end{equation}

\noindent where $Aut(B)$ is the  group of automorphisms of $B$ and $U(B)$ denotes the subgroup given by the units of $B$, and, if we write $\prescript{g}{}{{b}} = \sigma(g)(b)$, then the following identities must hold: 

\begin{enumerate} 
\item $\prescript{x}{}{{(\prescript{y}{}{{b}})}} = \alpha(x,y) \prescript{xy}{}{{b}} \alpha(x,y)^{-1}$$\forall b \in B, x,y \in G$
\item $ \alpha(x,y)\alpha(xy,z)=\prescript{x}{}{\alpha(y,z)} \alpha(x,yz)$ $\forall x,y,z \in G$
\item $ \alpha(x,e) = \alpha(e,x) = 1$ $\forall x \in G$
\end{enumerate}

A map $\alpha$ that holds the previous conditions is called a $2-cocycle$.
\end{definition}

For every crossed system we can define what is known as a crossed product or twisted skew group ring.

\begin{definition}
Let $(G,B,\sigma,\alpha)$ be a crossed system. The crossed product, also called twisted skew group ring, of $B$ over $G$, $B^\alpha [ G; \sigma]$, is the free $B-$module generated by the elements $g\in G$, denoted by $\overline{g}$, with the inner operation

\begin{equation}
(a \overline{x}) \cdot (b \overline{y}) = a \prescript{x}{}{b} \alpha(x,y) \overline{xy} \forall a,b \in B, x,y \in G
\end{equation}

\noindent and which is extended by linearity.
\end{definition}

In this work, we will deal with a twisted skew group ring over a finite field.

\section{Code Checkable}
Let $G$ be a group, let $p$ be a prime number, and $\mathbb{F}_{p^n}$ be the finite field of $p^n$ elements, and let $(G,\mathbb{F}_{p^n},\sigma,\alpha)$ be a crossed system. Then, we know that $Aut(\mathbb{F}_{p^n}) = Gal(\mathbb{F}_p,\mathbb{F}_{p^n})\simeq C_n$ the cyclic group of order $n$, generated by the Frobenious map. In particular, we have that $\mathbb{F}_p$ is fixed by the action of $G$ and $\mathbb{F}_{p^n}^\alpha[G,\sigma]$ is an algebra over its prime field.

Let $\hat{G}= \mathbb{F}_{p^n}^* \times G$ and let 
\begin{align*}
    \hat{G} \times \hat{G} & \longrightarrow \hat{G} \\
    ((a,g),(b,h)) & \longmapsto (a \prescript{g}{}{b} \alpha(g,h) , gh)
\end{align*}

\begin{proposition} \label{hatGisgroup}
    $\hat{G}$ with the previous operation is a group. 
\end{proposition}
\begin{proof}
    It is close under the operation as $\mathbb{F}_{p^n}$ is an integral domain, and the associativity comes from the associativity of the product in a twisted skew group ring. 

    It has an identity element, given by $1_{\hat{G}} = (1_{\mathbb{F}_p^n}, 1_G)$ as for all $(x,y) \in \hat{G}$
    \begin{align*}
        (x,y) 1_{\hat{G}} = &  (x,y) (1_{\mathbb{F}_p^n}, 1_G) \\ = & (x \prescript{y}{}{1_{\mathbb{F}_p^n}}, y 1_G )  \\ = & (x,y)  \\ = & (  1_{\mathbb{F}_p^n}\prescript{1_G}{}{x},  1_G y)  \\ = & (1_{\mathbb{F}_p^n}, 1_G)  (x,y)  \\ = & 1_{\hat{G}} (x,y)
    \end{align*}
    Finally, the inverse element of $(x,y) \in \hat{G}$ is $(\prescript{y^{-1}}{}{} \left(\frac{1}{x\alpha(y,y^{-1})}\right),y^{-1})$, as 
    \begin{equation}
        (x,y)  (\prescript{y^{-1}}{}{} \left(\frac{1}{x\alpha(y,y^{-1})}\right),y^{-1}) = \left( \frac{x\alpha(y,y^{-1})}{x\alpha(y,y^{-1})} , y y^{-1}\right) = (1_{\mathbb{F}_{p^n}},1_G) = 1_{\hat{G}}
    \end{equation}
    and the other side is proven analogously.
\end{proof}
\begin{proposition}
    There is an epimorphism of $\mathbb{F}_p$-algebras over  $\mathbb{F}_p \hat{G} \longrightarrow \mathbb{F}_{p^n}^\alpha[G,\sigma]$.
\end{proposition}
\begin{proof}
The epimorphism $\psi :\mathbb{F}_p \hat{G}  \longrightarrow \mathbb{F}_{p^n}^\alpha[G,\sigma]$ given by $\psi(a \overline{(b,c)}) = ab \overline{c}$ for all $a\in \mathbb{F}_p$, $(b,c) \in \hat{G}$ and extended by linearity. It respects addition by definition, so we only have to prove that it respects product, and to do so we can focus on the basis elements

\begin{equation}
    \psi(\overline{(a,g)} \overline{(b,h)} ) = \psi(  (a \prescript{g}{}{b} \alpha(g,h) , gh) = a \prescript{g}{}{b} \alpha(g,h) \overline{gh} = ( a\overline{g} ) ( b \overline{h} ) = \psi(\overline{(a,g)}) \psi ( \overline{(b,h)}
\end{equation}
\end{proof}

To begin with, we recall the definition of checkable ideals.
\begin{definition}
A right ideal $I \le A$ is called checkable if there exists an element $v \in A$
such that
\[
I = \{\, a \in A \mid va = 0 \,\}
    = \operatorname{Ann}_r(v)
    = \operatorname{Ann}_r(Av).
\]
Here, checkable left ideals are defined analogously via the left annihilator of a principal
right ideal. A group algebra $K^\alpha G$ is called code-checkable if all right ideals of
$K^\alpha G$ are checkable. 
\end{definition}

In \cite{BorelloDeLaCruzWillems2022} the following result is proved.

\begin{corollary}[Corollary 3.2, \cite{BorelloDeLaCruzWillems2022}]
Let $\operatorname{char} K = p$ and let $B_0(G)$ be the principal $p$-block of $KG$. The following are equivalent:
\begin{enumerate}
    \item $G$ is $p$-nilpotent with a cyclic Sylow $p$-subgroup.
    \item $KG$ is a code-checkable group algebra.
    \item All right ideals in $B_0(G)$ are checkable.
\end{enumerate}
\end{corollary}

In \cite{Behajaina2024Twisted} a question regarding when twisted-skew G-code are code checkable is proposed.
\begin{proposition}[{\cite[Proposition~1.10]{Behajaina2024Twisted}}]
Let $\operatorname{char} K = p$. If $G$ is $p$-nilpotent with a cyclic Sylow $p$-subgroup,
then every left (right) ideal in a twisted group algebra is principal.
\end{proposition}

\begin{question}[{\cite[Question~1.11]{Behajaina2024Twisted}}]
Is there an extension of Proposition~1.10 to left ideals in twisted skew group rings?
\end{question}

To present the answer, we begin with some previous results. First of all, note that for all $(a,g) \in \hat{G}$ we have that
\begin{equation}
    (a,g) ^b = ( \prod_{k=0}^{b-1} \alpha(g,g^k)\prescript{g^k}{}{a}, g^b)
\end{equation}

\begin{lemma}
    $G$ is $p$-nilpotent if and only if $\hat{G}$ is  $p$-nilpotent
\end{lemma}
\begin{proof}
$(\Longrightarrow)$

    Suppose that $G$ is $p$-nilpotent, then all elements of order coprime with $p$ form a subgroup $H\leq G$. Then the set $\hat{H}=\mathbb{F}_{p^n}^* \times H$ is precisely the set of all elements whose order is coprime $p$ of $\hat{G}$.

    First, let $(a,b)\in \hat{G}$ of order a multiple of $p$ i.e. $(a,b) = p^c h$ with $gcd(h,p)=1$. Then $(a,b)^h = (a',b^h)$ is of order $p^c$, so $b^h$ is of order a power of $p$, and $b\not \in H$.

    In addition, if $(a,b) \in \hat{G}$ is of order coprime with $p$, by the same argument $b$ is of order coprime with $p$ and $(a,b) \in \hat{H}$.

    We have shown that $\hat{H}=\mathbb{F}_{p^n}^* \times H$ is the set of all elements of order coprime with $p$ of $\hat{G}$. The fact that it is a subgroup is proven analogously to \ref{hatGisgroup} using the fact that $H$ is a subgroup.

    $(\Longleftarrow)$

    Let $\pi : \hat{G} \longrightarrow G$ the projection on the second component. It is clearly a epimorphism of groups. Let $\hat{H}$  be the group form by all elements of order coprime with $p$. We will show that $\pi(\hat{H}) = H\leq G$ is a group form by all elements whose order coprime with $p$ in $G$. 
    
    First of all, as $\hat{H}$ is a subgroup, $H$ is also a group, so we only have to prove the $p-$nilpotent part. If $a\in G$ has order corpime with $p$, then $(1,a)$ has order coprime with $p$, so $a = \pi((1,a))  \in \pi(\hat{H}) = H$. Now, if $a\in H$, then it is the image of $(1,a) \in \hat{H}$, and $ord(a) = ord((1,a))$ so it has order coprime with $p$. 
\end{proof}

\begin{lemma}
$G$ has a cyclic Sylow $p$-subgroup if and only if $\hat{G}$ also. 
\end{lemma}
\begin{proof}
To begin with, note that $|\hat{G}| = | G | |\mathbb{F}_{p^n}^*| = |G| (p^n-1)$, so the maximal power of $p$ that divides $|\hat{G}|$ is also the maximal power that divides $|G|$. Also, from the proof of the previous theorem, if $(a,b)\in \hat{G}$ is of order $p^n$ then $b$ is of order $p^n$.  

$(\Longleftarrow)$

If $\hat{G}$ has a cyclic Sylow $p-$subgroup $\hat{P}$, then $\pi(\hat{P})\leq G$ is also cyclic, as it is the homomorphic image of a cyclic group. Let $(a,b)$ be it generator. Then $b$ is of order $p^n$ and $b\in \pi(\hat{P})$, so $|<b>|=p^n$ and therefore is a cyclic  Sylow $p-$subgroup of $G$.

$(\Longrightarrow)$

Suppose that $G$ has a cyclic Sylow $p-$subgroup $P$. Then all $p$-subgroups of $G$ are cyclic from the first Sylow theorem. Now, let $\hat{P}$ be the $p$-Sylow subgroup of $\hat{G}$. 

Let $(a,g),(b,g) \in \hat{P}$ with $a\not = b \in \mathbb{F}_{p^n}$. Then $(a,g) (b,g)^{-1} \in \hat{P}$, so
\[
(a,g) (b,g)^{-1} = (a,g) \left( \prescript{g^{-1}}{}{\left(\frac{1}{b\alpha(g,g^{-1})}\right)}, g^{-1} \right) = (ab^{-1},1_G) 
\]
So $(ab^{-1},1_G) \in \hat{P}$, but $ord((ab^{-1},1_G))$ is coprime with $p$, so the only possibility is $ord((ab^{-1},1_G))=1$ so $a=b$

Now, by hypothesis $\pi(\hat{P})$ is cyclic, and generated by some element $h\in G$ of order $p^t$ with $t\leq n$. Then from the previous computations, there is a unique $a_{i}$ for all $i=0,\cdots, t$  such that $(a_i,h^i)\in \hat{P}$. As $\pi$ is an epimorphism, then all elements of $\hat{P}$ are of this form, $t=n$ and $\hat{G}$ has a cyclic Sylow $p$-group
\end{proof}

\begin{proposition}
    If $G$ is $p$-nilpotent with a cyclic Sylow $p$-subgroup then The twisted group ring $\mathbb{F}_{p^n}^\alpha[G,\sigma]$ is code checkable 
\end{proposition}
\begin{proof}
    $(\Longrightarrow)$

    If $G$ is $p$-nilpotent with a cyclic Sylow $p$-subgroup then $\hat{G}$ is  $p$-nilpotent with a cyclic Sylow $p$-subgroup. Then $\mathbb{F}_p\hat{G}$ is code checkable. Then let $I\trianglelefteq_r \mathbb{F}_{p^n}^\alpha[G,\sigma]$ a right ideal, and let $\hat{I} = \pi^{-1}( I)$. Then $\hat{I}\trianglelefteq_r \mathbb{F}_p \hat{G}$ and therefore is code checkable, so $\hat{I} = \operatorname{Ann}_r(v)$ with $v\in \mathbb{F}_p \hat{G}$. Therefore $I=\pi(\hat{I}) = \pi(\operatorname{Ann}_r(v)) = \operatorname{Ann}_r(\pi(v))$
\end{proof}

This, in addition, also gives sufficient conditions for $\mathbb{F}_{p^n}^\alpha[G,\sigma]$ to have only principals left ideals. To do so, recall the well known result about Frobenius algebras.
\begin{lemma}
Let $A$ be a Frobenius algebra and let $L \le A$ be a left ideal.  
Set $I = \Ann_r(L)$. Then
\[
\Ann_\ell(I) = L.
\]
\end{lemma}

Now we can prove the generalization to our case. 

\begin{proposition}

Let $K^\alpha G$ be a code\!-checkable group algebra. Then every left ideal of $K^\alpha G$
is principal.
    
\end{proposition}
\begin{proof}
Since $K^\alpha G$ is code\!-checkable, every right ideal is checkable by definition.
By the previous result, in a Frobenius algebra,
a right ideal $I$ is checkable if and only if its left annihilator
$\Ann_\ell(I)$ is a principal left ideal. Hence
\[
\Ann_\ell(I) \ \text{is principal for every right ideal } I \le KG.
\tag{$\ast$}
\]

Now let $L \le KG$ be any left ideal.  Consider the right ideal
\[
I := \Ann_r(L).
\]
By the double annihilator property for Frobenius algebras (Proposition~2.2),
we have
\[
L = \Ann_\ell(\Ann_r(L)) = \Ann_\ell(I).
\]
Thus $L$ is the left annihilator of a right ideal.  By $(\ast)$,
$\Ann_\ell(I)$ is principal, and therefore $L$ is principal.

Hence every left ideal of $K^\alpha G$ is principal as in \cite{Behajaina2024Twisted} it is proven that it is Frobenius.
\end{proof}

\section{Abelian ideals in twisted group algebra}

First of all, recall some basic definitions of coding theory.
\begin{definition}
    Let $C_1, C_2 \subset \mathbb{F}_q^m$ codes are said to be equivalent if there exist a diagonal matrix $D = diag(\lambda_1, \cdots, \lambda_m) \in Mat_m(\mathbb{F}_q)$ and a permutation matrix $P$ such that  $C_1 = \{PDx; x \in C_2 \}$. They are said to be permutation equivalent if $D= Id_m$.
\end{definition}

In \cite{GARCIAPILLADO2019167} the following result regarding when ideals over group algebra are permutation equivalent to abelian group algebras.
\begin{theorem}[Theorem 1, \cite{GARCIAPILLADO2019167}]
    Let $C$ be a $G$-code over a finite field $K$ for a finite group $G$. If
\[
\dim_K(C) \le 3,
\]
then $C$ is permutation equivalent to an abelian group code.
\end{theorem}
We will generalize the result to the case of twisted group codes. To begin with, we need some previous results, for which in the same paper appears
\begin{lemma}[Lemma 1.3, \cite{GARCIAPILLADO2019167}]
    Let $K$ be a field and let $G$ and $H$ be two groups of the same order
$n < \infty$. Suppose that there exist two normal subgroups $N \triangleleft G$
and $F \triangleleft H$ such that $G/N \cong H/F$. If $N$ acts trivially on some
(left, right, or two-sided) ideal $I \subseteq KG$, then $I$ is permutation
equivalent to some (left, right, or two-sided) ideal of the ring $KH$.
\end{lemma}

Our generalization needs the following definition

\begin{definition}
    Let $K^\alpha G$ be a twisted group algebra with $G$ finite. We will say that a subgroup $N$ of $G$ acts by scalar on a subset $S\subset K^\alpha G$ if there exists a $\alpha-$ homomorphism $\lambda: N \longrightarrow K$ such that
    \begin{equation}
        \overline{u} x = \lambda(u) x \quad \forall x\in S, u\in N
    \end{equation}
\end{definition}

Now we can state that
\begin{lemma} \label{LemmaScalar}
    Let $K^\alpha G$ be a twisted group algebra with $G$ finite. Let $H$ be a group of the same order as $G$. Suppose that there exist two normal subgroups $N \triangleleft G$
and $F \triangleleft H$ such that $G/N \cong H/F$. If $N$ acts by scalar on some
(left, right, or two-sided) ideal $I \subseteq K^\alpha G$, then $I$ is
equivalent to some (left, right, or two-sided) ideal of the ring $KH$.
\end{lemma}
\begin{proof}
Let $s = |G/N|$ and denote by $g_1, \ldots, g_s$ a complete set of representatives
of $G$ modulo $N$. Thus
\[
G/N = \{ g_1N, \ldots, g_sN \}
\quad \text{and} \quad
G = \bigcup_{i=1}^s g_i N .
\]

Fixing the isomorphism $f \colon G/N \to H/F$, we can choose a representative
system $\{h_i\}_{i=1}^s$ of the group $H$ modulo $F$ such that
\[
f(g_iN) = h_iK, \quad i = 1, \ldots, s.
\]

For each pair of indices $i,j \in \{1, \ldots, s\}$, let $k(i,j)$ be defined by the
equality
\[
g_i N g_j N = g_{k(i,j)} N.
\]
Then we have
\[
g_i g_j = g_{k(i,j)} u_{ij} \quad \text{for some } u_{ij} \in N,
\]
and similarly
\[
h_i h_j = h_{k(i,j)} v_{ij} \quad \text{for some } v_{ij} \in F.
\]

Since $|N| = |F| = n/s$, we can fix a one-to-one mapping
\[
\tau \colon N \to F.
\]
Define $\varphi \colon G \to H$ as follows: for an arbitrary $x \in G$, the element
$x$ belongs to exactly one coset $Ng_i$ of $G$ modulo $N$. Set
\[
\varphi(x) = \frac{\alpha(xg_i^{-1}, g_i x^{-1})}{\alpha(xg_i^{-1}, g_i)\lambda(xg_i^{-1})}\overline{ h_i \, \tau( xg_i^{-1})}
\]
Clearly, $\varphi$ is a well-defined one-to-one map.

Let $x=\sum_{g\in G} a_g \overline{g} \in I$ and let $\lambda$ be associated $\alpha-$homomorphism of $N$ over $I$. Then, for all $u\in N$ we have that
     \begin{equation}
         \lambda(u)\sum_{g\in G} a_g \overline{g} =  \lambda(u)x = \overline{u}x = \sum_{g\in G} a_g \alpha(u,g) \overline{ug}
     \end{equation}
     So $a_g=\frac{1}{ \lambda(u)}a_{u^{-1}g }\alpha(u,u^{-1}g)$, or equivalently $a_g=\frac{1}{ \lambda(u^{-1})}a_{ug }\alpha(u^{-1},ug)$. In addition, we have that 
     \begin{equation}
         \alpha(u^{-1},ug) = \frac{\alpha(u^{-1},u)}{\alpha(u,g)}
     \end{equation}
     and therefore we have the equality
     \begin{equation}
         a_{ug} =  \lambda(u)\frac{\alpha(u,g)}{\alpha(u^{-1},u)} a_g
     \end{equation}

Then $x=\sum_{i=1}^s a_{g_i} \sum_{u\in N}  \frac{ \lambda(u)\alpha(u,g)}{\alpha(u^{-1},u)} \overline{ug_i}$. Now, note that for all $i$ we have that
\begin{align*}
    \varphi\left(\sum_{u\in N}  \frac{ \lambda(u)\alpha(u,g)}{\alpha(u^{-1},u)} \overline{ug_i}\right) & = \sum_{u\in N}  \frac{ \lambda(u)\alpha(u,g)}{\alpha(u^{-1},u)} \varphi(\overline{ug_i}) \\ & = \sum_{u\in N}  \frac{ \lambda(u)\alpha(u,g)}{\alpha(u^{-1},u)} \frac{\alpha(u^{-1},u) }{\alpha(u,g) \lambda(u)} \overline{ h_i \tau(u)} \\ & = h_i \sum_{u \in N} \tau(u) \\ & = h_k F_{\sum}
\end{align*}
Also, for all $h\in H$, we have that $h=h_j v$ for a unique $h_j$ and $v\in F$. We have that
\begin{align*}
    \overline{h}\varphi(x) & =  \overline{h_jv} \sum_{i=1}^s a_{g_i} \overline{h_i} F_{\sum}  \\ & =  \sum_{i=1}^s a_{g_i} \overline{h_jv}\overline{h_i} F_{\sum}  \\ & =   \sum_{i=1}^s a_{g_i} \overline{h_jv}\overline{h_i} F_{\sum}
    \\ & = \sum_{i=1}^s a_{g_i} \overline{h_jh_jv'}F_{\sum}
     \\ & = \sum_{i=1}^s a_{g_i} \overline{h_{k(j,i)}v''v'}F_{\sum}
     \\ & = \sum_{i=1}^s a_{g_i} \overline{h_{k(j,i)}}F_{\sum} 
     \\ & = \sum_{i=1}^s a_{g_i} F_{\sum} \overline{h_{k(j,i)}}
     \\ & = \varphi\left(\sum_{i=1}^s a_{g_i} \sum_{u\in N}  \frac{\alpha(u,g_{k(j,i)})}{\alpha(u^{-1},u) \lambda(u)}\overline{ug_{k(j,i)}}\right) \in \varphi(I)
\end{align*}
a similar calculation lead to $\varphi(x)\overline{h} \in \varphi(I)$ and we have the result

\end{proof}

Now, we will start the proof of the following theorem. r

\begin{theorem}
    Let $K^\alpha G$ be a twisted group algebra whit $G$ finite. Let $C$ be a twisted $G$-code over a finite field $K$ for a finite group $G$. If
\[
\dim_K(C) \le 3,
\]
then $C$ is an abelian group code.
\end{theorem}

We first start with dimension $1$.
\begin{theorem}
    Let $K^\alpha G$ be a twisted group algebra whit $G$ finite. If there exists a $1$ dimensional ideal then $K^\alpha G \simeq K G$. In particular, all $1$ dimensional ideal are permutation equivalent to an abelian group code.
\end{theorem}
\begin{proof}
    If there exist $I\triangleleft K^\alpha G $, then $I= Kv$ with $v\in K^\alpha G$. Then, for all $g\in G$, we have that $\overline{g} v = \lambda(g) v$ with $\lambda(g) \in K^*$. Therefore, we have
    \begin{equation}
        \lambda(gh) \alpha(g,h)v  =  \alpha(g,h) \overline{ gh} v = (\overline{g} \cdot \overline{h} ) =\overline{g} ( \overline{h} v) = \lambda(g) \lambda(h) v
    \end{equation}
    And therefore 
    \begin{equation}
        \alpha(g,h) = \lambda(g) \lambda(h) \lambda(gh)^{-1}
    \end{equation}
    so $\alpha $ is a $2-$coboundary and therefore $K^\alpha G \simeq K G$
\end{proof}

This will help us to prove that
\begin{theorem}
    Let $K^\alpha G$ be a twisted group algebra whit $G$ finite, and let $I\triangleleft K^\alpha G$ of dimension $2$. Then $I$ is equivalent to some ideal in a commutative group ring
over $K$ .
\end{theorem}
\begin{proof}
    If $I$ is not simple, then there exist a one dimensional ideal over $K^\alpha G$, and due to the previous lemma we have that $K^\alpha G \simeq KG$ and the result comes form the non twisted case.  
    
    Then by Schur's lemma  its endomorphism ring is a division ring $D$, that, as $K$ and $G$ is finite, $I$ is finite and $D$ is a field. Therefore, multiplication by $K^\alpha G$ defines an $K-$homomorphism over $D$, which we will denote by $\psi$. 
    
    As $D$ is commutative,  $G'$ acts by scalar. To see this, note that $x=\overline{g} \cdot \overline{h} \cdot \overline{g^{-1}} \cdot \overline{h^{-1}}$ satisfy $\psi(x) = \psi(g) \psi(h) \psi(g^{-1}) \psi(h^{-1}) = \psi(g) \psi(h) \psi(g) ^{-1} \psi(h)^{-1} = Id$, and 
    
    \[\psi(x) = \psi( \alpha(g,h)\alpha(gh,g^{-1})\alpha(ghg,h^{-1})\psi ( \overline{ghg^{-1}h^{-1})}\]
    so we have that
    
    \[\psi\left(\overline{ghg^{-1}h^{-1}}\right) ( a) = \frac{1}{\alpha(g,h)\alpha(gh,g^{-1})\alpha(ghg,h^{-1})} a\] 
    
    and $G'$ acts by scalar over $I$. Now, if we take $H=G' \times G/G'$, $N=F=G'$ we are under the conditions for lemma \ref{LemmaScalar} 

\end{proof}

Now, for dimension $3$ we need the following lemma from \cite{GARCIAPILLADO2019167}

\begin{lemma}[Lemma 2.3 \cite{GARCIAPILLADO2019167}]
    Let \( R \) be an \( F \)-algebra. Suppose that \( I \triangleleft R \) and
\(\dim_F(I) = 3\). If \( M \) is a two-dimensional simple submodule of \( IR \),
then \( M \) is a fully characteristic submodule of \( IR \)
(i.e.\ \( f(M) \subseteq M \) for any \( f \in \operatorname{End}(IR) \)),
in particular, \( M \subseteq R \).

If \( IR/N \) is a two-dimensional simple factor module of \( IR \) for some
submodule \( N \), then \( N \) is a fully characteristic submodule of \( IR \),
in particular, \( N \triangleleft R \).

\end{lemma}

Now, we can prove

\begin{theorem}
    Let $K^\alpha G$ be a twisted group ring with $G$ finite and let $I\triangleleft K^\alpha G$ be an ideal of dimension $3$ over $K$. Then $I$ is equivalent to an abelian group code. 
\end{theorem} 
\begin{proof}
    If $I$ is simple, then the same argument as in dim $2$ remain valid. If $I$ has a 1 dimensional ideal then it is also true as then $K^\alpha G \simeq KG$ and the group ring case prove it. So it remain to prove the case when it has a $2$ dimensional sub ideal.  

    If $I$ has a $2-$dimensional sub ideal $M \triangleleft I$, let $v_1,v_2,v_3 \in I$ be a base of $I$ with $v_1,v_2 \in M$. Then we have that $\overline{g} v_1,\overline{g} v_2 \in M $ for all $g\in G$ and $\overline{g} v_3 = \beta_1(g) v_1 + \beta_2(g) v_2 + \beta_3(g) v_3$. If we see the component of $v_3$ we have that $\alpha(g,h)\overline{gh} v_3 = \overline{g} \cdot \overline{h} v_3$ so $\alpha(g,h) \beta_3(gh) =  \beta_3(g)  \beta_3(h)$ with $\beta_3 \in \mathbb{F}_p^n$ for all $g,h \in G$. 

    If $\beta_3 \not = 0 $ for all $g$ then $\alpha$ is a $2-$coboundary and $K^\alpha G \simeq KG$ so the statement is true. Now, if $\beta_3(g) =0$ for some $g$, then as $\alpha(g,h) \beta_3(gh) =  \beta_3(g)  \beta_3(h) = 0$ and $\alpha(g,h)\not = 0$ we have that $\beta_3(gh)=0$ for all $h$, so $\beta_3(h) =0$ for all $h\in G$, so $1v_3 \in <v_1,v_2> = M$ absurd as they are linearly independent
\end{proof}

\section{Bounds on distance and dimension of a group code} 
The next result was first proved for group codes in \cite{Borello2022Ideals}. An extension to twisted group codes has been given proved in \cite{WillemsTwistedToAppear}, and for twisted skew group ring, and the presented version in this paper, in \cite{Behajaina2024Twisted}. 

\begin{theorem} [Theorem 2.4, \cite{Behajaina2024Twisted}]
    If $0 \neq C \le K^\alpha[G,\sigma] = R$ is a twisted skew group code
of minimum distance $d(C)$, then
\[
|G| \le d(C)\cdot \dim C.
\]

\end{theorem}
In addition, in \cite{Borello2022Ideals} group codes where the equality hold have been charactericed.
\begin{theorem}[Theorem 2.10., \cite{Borello2022Ideals}]
    A $G$-code $C$ satisfies $d(C)\cdot \dim C = |G|$ if and only if there exist
$H \le G$ and $c \in K H$ such that $|H| = d(C)$, $c K H$ has dimension $1$,
and $C = c K G$.
\end{theorem}

Now, we will generalize their result to the case of twisted group algebras. Following \cite{FengHollmannXiang2019} we have

\begin{definition}
Let $G$ be a group, and $S$ be a nonempty subset of $G$. A sequence $g_1,\dots,g_t$ in $G$ has right $S$-rank $t$ if
\[
S g_i := \{ s g_i \mid s \in S \}
\]
is not contained in $\bigcup_{j<i} S g_j$ for all $i \in \{2,\dots,t\}$.
\end{definition}

Let $R$ be a finite $G-$graded $K-$algebra with $G$ a finite group. For any $f \in R$, we denote by $T_f \colon  C \to  C$ the map
$v \mapsto f v$. Then the following results from \cite{Borello2022Ideals} are true in our case

\begin{lemma} \label{lemmasequence} [Lemma 2.2, \cite{Borello2022Ideals}]
Let $0 \neq f \in R$ and $S = \supp(f)$. If there exists a sequence in $G$
with right $S$-rank $t$, then
\[
\dim f  R = \rank_K(T_f) \ge t.
\]
\end{lemma}

\begin{theorem} \label{teoremaSuppRank} [Lemma 2.4, \cite{Borello2022Ideals}]
For any $f \in  R$, define $T_f \colon  R \to  R$ by $v \mapsto f v$.
Then
\[
|\supp(f)| \cdot \rank_K(T_f) \ge |G|.
\]
\end{theorem}

\begin{corollary} \label{Cota inferior}
For any nonzero ideal $C\triangleleft R$, we have
\[
d(C)\cdot \dim C \ge |G|.
\]
In particular,
\[
2\sqrt{|G|} \le d(C) + \dim C \le |G| + 1.
\]
\end{corollary}

\begin{proof}
Let $f \in C$ be such that $d(C) = |\supp(f)|$.
Then the subcode $f  R$ has the same minimum distance as $C$,
but possibly a smaller dimension.
Since $\dim(f  R) = \rank_K(T_f)$,
the first claim follows from Theorem \ref{Cota inferior}.
The second claim is a consequence of the AM--GM inequality
and the Singleton bound.
\end{proof}

Note that group ring, twisted group ring and twisted skew group ring  are $G$-graded $K$-algebras, so all the previous results also apply to our case. In particular

\begin{lemma}
    If $R=\otimes_{g\in G} R_g$ is a strongly $G-$graded $K-$algebra with $\dim_K R_g =1$, then $R\simeq K^\alpha G$ for some  $2-$cocycle $\alpha :G\times G \longrightarrow K$.
\end{lemma}
So the interesting case for our results in a code perspective are all twisted group rings. Now we can prove the following result

\begin{theorem}
Let $C$ be a twisted $G-$code
\[
d(C)\cdot \dim C = |G|
\]
if and only if there exist $H \le G$ and $c \in R|_H$ such that
$|H| = d(C)$, $c K^\alpha H$ has dimension $1$, and $C = c K^\alpha G$.
\end{theorem}

\begin{proof}
Assume that there exist a subgroup $H \leq G$ and an element $c \in K^{\alpha}H$ satisfying the hypotheses of the statement. Then $C$ is the module induced from $cK^{\alpha}H$, regarded as a $K^{\alpha}H$-module, and it splits as a direct sum of precisely $[G : H]$ direct summands. Note that $\dim_K K^\alpha H = \dim_K K^\alpha Hg$ for all $g\in G$ so each summand is of dimension $1$.

Thus $\dim C = [G:H]$ and
\[
d(C)\cdot \dim C = |H|\cdot [G:H] = |G|.
\]

To prove the other direction, suppose that $d(C)\cdot \dim C = |G|$.
Let $c \in C$ be an element such that $\wt(c) = d(C)$ and consider
\[
C_0 := c K^\alpha G \le C.
\]
We have that $d(C_0) = d(C)$, and by Corollary \ref{Cota inferior}
\[
|G| \le d(C_0)\cdot \dim C_0
= d(C)\cdot \dim C_0
\le d(C)\cdot \dim C
= |G|.
\]
It follows that equality holds, and in particular $\dim C_0 = \dim C$. Therefore
\[
C = C_0 = c K^{\alpha}G.
\]

Now we will prove that there exist such $H$ as the theorem state. 
Let $H := \supp(c)$, so that $|H| = d(C)$.
Following the arguments and notation from the proof of Theorem \ref{Cota inferior}
with $f = c$, we obtain
\[
t = \frac{|G|}{|H|}
\quad\text{and}\quad
G = H g_1 \,\dot{\cup}\, \cdots \,\dot{\cup}\, H g_t
\]
for some $g_1,\dots,g_t \in G$.

Replacing $c$ by the minimum weight codeword $c h^{-1} \in C$ for some
$h \in \supp(c)$, we may assume that $1 \in H$.
Then for any $h \in H$ we have $h \in H \cap H h$.
We claim that this implies $H = H h$.

Suppose by contradiction that $H h \nsubseteq H$.
Choose a maximal sequence $h_1,\dots,h_r \in G$ with $r \ge 2$ such that
$h_1 = 1$, $h_2 = h$, and
\[
H h_i \nsubseteq \bigcup_{1 \le j < i} H h_j
\quad\text{for } 2 \le i \le r.
\]
By maximality,
\[
G = \bigcup_{1 \le j \le r} H h_j.
\]
By definition of $t$ as in the proof of Theorem \ref{Cota inferior}, we have $r \leq t $.
However, since $H \cap H h \neq \varnothing$, we obtain
\[
r > \frac{|G|}{|H|} = t,
\]
a contradiction, so $H h = H$ for all $h \in H$, and therefore $H \cdot H \subseteq H$.
Since $H$ is finite, this implies that $H$ is a subgroup of $G$ and
$t = [G:H]$.

Finally, using the decomposition
\[
C = c K^\alpha G
= c K^\alpha \Bigl( \bigsqcup_{i=1}^t H g_i \Bigr)
= \bigoplus_{i=1}^t (c K^\alpha Hg_i) ,
\]
in addition, as $\dim_K(cK^\alpha H )=\dim_K(cK^\alpha H g_i )$ because $\overline{g_i}$ is a unit, we obtain
\[
\dim C = t \cdot \dim(c K^\alpha H).
\]
By assumption,
\[
\dim C = \frac{|G|}{d(C)} = t,
\]
and hence $\dim(c K^\alpha H) = 1$.
This proves the claim.
\end{proof}

\section{Conclusion}
We have solved an open question concerning when a  twisted skew $G-$code is code checkable. We have also generalized several results about group ring to the case of twisted group ring, such as that all twisted $G-$codes of dimension up to 3 are equivalent to an abelian group code, and also generalize a bound on the dimension and the minimum distance of a twisted $G$-code.

\section{Acknowledgments}
This research was supported by the Spanish Ministry of Science, Innovation and Universities under the FPU 2023 grant program.

\bibliographystyle{plain}
\bibliography{Article}

@article{Berman1967,
  author  = {Berman, S. D.},
  title   = {On the Theory of Group Codes},
  journal = {Kibernetika},
  volume  = {3},
  pages   = {31--39},
  year    = {1967}
}

@article{BorelloDeLaCruzWillems2022,
  author       = {Borello, M. and de la Cruz, J. and  Willems, W.},
  title        = {On Checkable Group Codes},
  journal      = {Journal of Algebra and Its Applications},
  volume       = {21},
  year         = {2022},
  number       = {},
  pages        = {},
  doi          = {},
}

@article{BoucherGeiselmannUlmer2007,
  author  = {Boucher, D. and Geiselmann, W. and Ulmer, F.},
  title   = {Skew-Cyclic Codes},
  journal = {Applicable Algebra in Engineering, Communication and Computing},
  volume  = {18},
  pages   = {379--389},
  year    = {2007}
}

@article{Behajaina2024Twisted,
  author    = {Behajaina, A. and  Borello, M. and  de la Cruz J.},
  title     = {Twisted skew {G}-codes},
  journal   = {Designs, Codes and Cryptography},
  volume    = {92},
  pages     = {1803--1821},
  year      = {2024},
  doi       = {10.1007/s10623-024-01367-0},
  url       = {https://doi.org/10.1007/s10623-024-01367-0},
  publisher = {Springer}
}

@article{Borello2022Ideals,
  author  = {Borello, M. and  Willems, W. and Zini G.},
  title   = {On ideals in group algebras: an uncertainty principle and the {S}chur product},
  journal = {Forum Mathematicum},
  year    = {2022},
  doi     = {10.1515/forum-2022-0064},
  url     = {https://doi.org/10.1515/forum-2022-0064}
}

@article{deLaCruzWillems2021,
  author  = {de la Cruz, J. and Willems, W.},
  title   = {Twisted Group Codes},
  journal = {IEEE Transactions on Information Theory},
  volume  = {67},
  pages   = {5178--5184},
  year    = {2021}
}

@article{Hamming1950,
  author  = {Hamming, R. W.},
  title   = {Error Detecting and Error Correcting Codes},
  journal = {The Bell System Technical Journal},
  volume  = {26},
  pages   = {147--160},
  year    = {1950}
}

@book{Karpilovsky1989,
  author    = {Karpilovsky, G.},
  title     = {The Algebraic Structure of Crossed Products},
  series    = {North-Holland Mathematics Studies},
  volume    = {142},
  publisher = {Elsevier},
  year      = {1989}
}

@article{FengHollmannXiang2019,
  author  = {Feng, T. and Hollmann, H. D. and Xiang, Q.},
  title   = {The Shift Bound for Abelian Codes and Generalizations of the Donoho--Stark Uncertainty Principle},
  journal = {IEEE Transactions on Information Theory},
  volume  = {65},
  number  = {8},
  pages   = {4673--4682},
  year    = {2019},
  doi     = {10.1109/TIT.2019.2911681}
}

@article{GARCIAPILLADO2019167,
author = {Garcia C. and González S. and V Markov V. and Markova O. and Martínez C.},
title = {Group codes of dimension 2 and 3 are abelian},
journal = {Finite Fields and Their Applications},
volume = {55},
pages = {167-176},
year = {2019},
issn = {1071-5797},
doi = {https://doi.org/10.1016/j.ffa.2018.09.009},
url = {https://www.sciencedirect.com/science/article/pii/S1071579718301242}
}

@misc{Grassl,
  author       = {Grassl, M.},
  title        = {Bounds on the Minimum Distance of Linear Codes},
  howpublished = {\url{http://www.codetables.de}},
  note         = {Online; accessed 23 December 2025}
}

@article{Shannon1948,
  author  = {Shannon, C. E.},
  title   = {A Mathematical Theory of Communication},
  journal = {The Bell System Technical Journal},
  volume  = {27},
  number  = {3},
  pages   = {379--423},
  year    = {1948},
  month   = jul
}

@article{WillemsTwistedToAppear,
  author  = {Willems, W.},
  title   = {Codes in twisted group algebras},
  journal = {To appear},
  note    = {To appear}
}

\end{document}